\documentclass{amsproc}
\usepackage{eurosym}
\usepackage{amssymb}
\usepackage{amsfonts}
\usepackage{cmtt}
\usepackage{tikz}
\usepackage[all,cmtip]{xy}

\theoremstyle{plain}

\newtheorem{theorem}{Theorem}[section]
\newtheorem{question}[theorem]{Question}

\newtheorem{definition}[theorem]{Definition}

\newtheorem{lemma}[theorem]{Lemma}

\newtheorem{proposition}[theorem]{Proposition}
\newtheorem{remark}[theorem]{Remark}

\numberwithin{equation}{section}

\begin{document}

\title[Cox rings]{Cox rings of du Val singularities}

\author{Laura Facchini} 
\address{
Department of Mathematics\\
University of Trento\\
Via Sommarive 14\\
38123 Povo (TN), Italy}
\email{laura.facchini@unitn.it}

\author{V\'\i ctor Gonz\'alez-Alonso} 
\address{
Department of Applied Mathematics I\\
Universitat Polit\`ecnica de Catalunya\\
Av. Diagonal 647\\
08028 Barcelona, Spain}
\email{victor.gonzalez-alonso@upc.edu}

\author{Micha{\l} Laso\'{n}} 
\address{
Institute of Mathematics of the Polish Academy of Sciences\\
\'{S}niadeckich 8, 00-956 Warszawa, Poland\\
Theoretical Computer Science Department\\
Faculty of Mathematics and Computer Science\\
Jagiellonian University, S. \L ojasiewicza 6, 30-348 Krak\'{o}w, Poland}
\email{michalason@gmail.com}



\thanks{The second author completed this work supported by grant AP2008-01849 from the Spanish Ministerio de Educaci\'on.}
\thanks{The third author is supported by the grant of Polish MNiSzW N N201 413139.} 

\begin{abstract}
In this note we introduce Cox rings of singularities and explicitly compute them in the case of du Val singularities $\mathbb{D}_n,\mathbb{E}_6,\mathbb{E}_7$ and $\mathbb{E}_8$.
\end{abstract}

\maketitle

\section{Introduction}

In the study of the intrinsic geometry of projective varieties it is very useful to consider all of its possible projective embeddings. This leads, for example, to the study of the Picard group of the variety, considering all possible line bundles on it. However, on the other hand we lose the concept of {\em coordinate ring}. In the case of toric varieties, this problem was solved by Cox \cite{Cox} by considering the {\em total coordinate ring} of the variety. Cox's construction was generalized by Hu and Keel in \cite{HuKeel} for varieties with free, finitely generated Picard group, and they called this ring the {\em Cox ring} of the variety.

The Cox ring of a variety is closely related to its birational geometry, specially to its various GIT presentations and small modifications. For example, if the Cox ring of a variety is finitely generated, then Mori's Minimal Model Program can be carried out for any divisor (hence the name {\em Mori Dream Space} for these varieties). Some examples of varieties, except toric ones, whose Cox ring has been explicitly computed are homogeneous varieties, Del Pezzo surfaces and some blow-ups of projective spaces in finitely many points (see \cite{LafaVela}).

In this note, we generalize this construction to study resolutions of surface singularities. More precisely, we define the Cox ring of a surface singularity under some hypothesis on the relative Picard group of its desingularization, and then compute explicitly the Cox ring of du Val singularities. We focus on these singularities because the Picard group of their desingularizations is easy to describe, so many computations can be done explicitly. Moreover, since they are the most basic surface singularities, they constitute a natural starting point for computing Cox rings of other singularities. For a short introduction about these singularities and its basic properties, the reader is referred to \cite{BHPVV}, sections III.3 to III.7.

The paper is organized as follows. First, in section 2, we introduce Cox rings by stating our definition and summarizing some general properties. In section 3 we compute the case of the $\mathbb{A}$-type singularities, which are toric varieties and hence its Cox ring is known. After that, in section 4, using the results from section 3 as a guide, we compute the Cox ring of the $\mathbb{D}_n$ singularities. This is the longest section in the paper, and contains full and detailed proofs of all the intermediate steps. Finally, in section 5, we state the basic results in the case of the three $\mathbb{E}$-type singularities. We have omitted the proof of these results because they are very similar to the proofs from section 4 and they would make the paper much longer (each of the three cases need specific proofs at some point).

{\em Note:} All the varieties considered in the paper are defined over $\mathbb{C}$.

\section{Generalities on Cox rings.}

Our aim is to compute the Cox ring of all du Val surface singularities. First of all, we need to define such a ring, and we do it for any (normal) surface singularity. 
\begin{definition} \label{def-Cox}
Let $(X,O)$ be a normal surface singularity, i.e., $X$ is a normal surface whose only singular point is $O$, and let $\pi: \widetilde{X} \longrightarrow X$ be its minimal desingularization. Assume moreover that the relative Picard group of $\pi$, ${\rm Pic}(\widetilde{X}/X) = {\rm Pic}(\widetilde{X})/\pi^*{\rm Pic}(X)$ is free and finitely generated. We define the {\em Cox ring} of the singularity as the ring $${\rm Cox}(X) = \bigoplus_{L \in {\rm Pic}(\widetilde{X}/X)} H^0(\widetilde{X},L).$$
\end{definition}

\begin{remark}
${\rm Cox}(X)$ is naturally graded by ${\rm Pic}(\widetilde{X}/X)$. Its multiplicative structure depends on the choice of line bundles representing isomorphism classes $L \in {\rm Pic}(\widetilde{X})$, but two different choices give (non-canonically) isomorphic rings.
\end{remark}

In our case, du Val singularities are rational, which implies that $${\rm Pic}(\widetilde{X}) \cong \pi^*{\rm Pic}(X) \oplus H^2(E,\mathbb{Z})$$ where $E = \pi^{-1}(O)$ is the fibre over the singular point. Moreover, $E = \bigcup_{i=1}^n E_ i$ is a connected union of (-2)-curves $E_i$ (i.e. each $E_i$ is a smooth  rational curve with self-intersection $(E_i)^2 = -2$) intersecting transversely at one point at most (i.e. $E_i \cdot E_j = 0,1$). Therefore, we get the following chain of isomorphisms $${\rm Pic}(\widetilde{X}/X) \cong H^2(E,\mathbb{Z}) \cong \bigoplus_{i=1}^n H^2(E_i,\mathbb{Z}) \cong \bigoplus_{i=1}^n \mathbb{Z}\langle c_1(\mathcal{O}_{E_i}(1))\rangle \cong \mathbb{Z}^n.$$ We call the composition $\delta: {\rm Pic}(\widetilde{X}/X) \longrightarrow \mathbb{Z}^n$ the {\em relative (multi)degree}, because its components are given by $\delta_i(L) = \deg(L_{|E_i})$ for any $L \in {\rm Pic}(\widetilde{X})$. Furthermore, any (isomorphism class of a) relative line bundle $L \in {\rm Pic}(\widetilde{X}/X)$ is uniquely determined by its multidegree $\delta(L)$. This nice description of the relative Picard group is the reason to focus on du Val (rational) surface singularities.

Before going deeper into the study of Cox rings of du Val singularities, we would like to remind some general properties of Cox rings (suitably adapted to our case), mainly those related with GIT and birational models. We are following \cite{LafaVela}, where the interested reader could find more detailed explanations.

First of all, ${\rm Cox}(X)$ is endowed with a natural action of the algebraic torus $T_X = {\rm Hom}({\rm Pic}(\widetilde{X}/X),\mathbb{C}^*)$. Explicitly, if $L_1,\ldots,L_r$ is a basis of ${\rm Pic}(\widetilde{X}/X) \cong \mathbb{Z}^r$ so that $T_X \cong \mathbb{C}^{*r}$, and $x \in H^0(\widetilde{X},L)$, with $L \cong L_1^{a_1}\otimes\cdots\otimes L_r^{a_r}$, then the action is given by $$(t_1,\ldots,t_r)\cdot x = t_1^{a_1}\cdots t_r^{a_r} x.$$

This action extends to an action on the affine scheme $\overline{X} = {\rm Spec}({\rm Cox}(X))$. Therefore, whenever ${\rm Cox}(X)$ is finitely generated, we can use it to do GIT. So, from now on we will assume that ${\rm Cox}(X)$ is finitely generated. Let $L$ be a line bundle on $\widetilde{X}$ and set $$R_L = \bigoplus_{m=0}^{\infty} {\rm Cox}(X)_{L^m} = \bigoplus_{m=0}^{\infty} H^0(\widetilde{X},L^m).$$ The inclusion $R_L \subseteq {\rm Cox}(X)$ induces a rational map $$\pi_L: \overline{X} \dashrightarrow X_L = {\rm Proj}(R_L)$$ which is constant on $T_X$-orbits. In particular, if we take $L$ to be a very ample line bundle on $\widetilde{X}$ then $X_L\cong\widetilde{X}$. If we take $L$ to be the pull-back of a very ample line bundle on $X$, instead, we recover $X_L \cong X$. Therefore, quotients of $\overline{X}$ give some insight into the geometry of $X$.

As explained in \cite{LafaVela}, $L$ induces a linearization of the trivial bundle on $\overline{X}$ extending the action of $T_X$, and the GIT quotient of $\overline{X}$ via this linearization is precisely $X_L$. This quotient is a good geometric quotient of the set of semistable points of $\overline{X}$, whose complement is the subvariety associated to the {\em irrelevant ideal} associated to $L$: $$J_L=\sqrt{(H^0(\widetilde{X},L))} \subseteq R_L.$$ This fact shows that, in our case, both the singularity $X$ and its desingularization $\widetilde{X}$ can be recovered from ${\rm Cox}(X)$ and some combinatorial data in ${\rm Pic}(\widetilde{X}/X)$ (namely the set of very ample line bundles).

\section{The case of $\mathbb{A}$ singularities: an inspiration.}

In order to get our first candidates for the Cox rings of (affine) $\mathbb{A}-\mathbb{D}-\mathbb{E}$ singularities, we look first at the case of $$\mathbb{A}_n = \{xy - z^{n+1}=0\} \subset \mathbb{C}^3,$$ which are toric varieties. In this case, the exceptional curve of the minimal desingularization is a chain of $n$ (-2)-curves $E_1,\ldots,E_n$, with intersection form $$E_i \cdot E_j = \begin{cases} -2 & \mbox{if $i = j$} \\ 1 & \mbox{if $|i-j|=1$} \\ 0 & \mbox{otherwise.} \end{cases}$$ The dual graph of the resolution is therefore: 

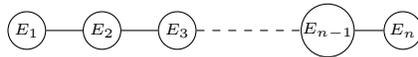
\begin{center}
\begin{tikzpicture}[scale=1]
\tikzstyle{every node}=[draw,circle,fill=white,minimum size=5mm,inner sep=0pt]
\path (0:0cm) node (e1) {\tiny $E_1$};
\path (0:1cm) node (e2) {\tiny $E_2$};
\path (0:2cm) node (e3) {\tiny $E_3$};
\path (0:4cm) node (en1) {\tiny $E_{n-1}$};
\path (0:5cm) node (en) {\tiny $E_n$};
\draw 
(e1) -- (e2)
(e2) -- (e3)
(en1) -- (en);
\draw[dashed](e3) -- (en1);
\end{tikzpicture}
\end{center}

Hence, in this case ${\rm Pic}(\widetilde{\mathbb{A}_n}/\mathbb{A}_n) \cong \mathbb{Z}^n$, and it is known (as for all toric varieties) that ${\rm Cox}(\widetilde{\mathbb{A}_n}) = \mathbb{C}[x_1,y_1,\ldots,y_n,x_n]$ with degrees
{\small
\begin{eqnarray*}
d(x_1) & = & e_1 = (1,0,\ldots,0), \\
d(y_1) & = & -2e_1+e_2 = (-2,1,0,\ldots,0), \\
d(y_k) & = & e_{k-1}-2e_k+e_{k+1} = (0,\ldots,0,1,-2,1,0,\ldots,0), \quad \mbox{ for } k=2,\ldots,n-1, \\
d(y_n) & = & e_{n-1}-2e_n = (0,\ldots,0,1,-2)\\
d(x_n) & = & e_n = (0,\ldots,0,1),
\end{eqnarray*}
}
(where $\{e_1,\ldots,e_n\}$ stands for the standard basis of $\mathbb{Z}^n$).

Thus, the GIT presentation of $\widetilde{\mathbb{A}_n}$ (or $\mathbb{A}_n$, depending on the chosen linearization) is given by the action of the torus $T = {\rm Hom}({\rm Pic}(\widetilde{\mathbb{A}_n}/\mathbb{A}_n),\mathbb{C}^*) \cong (\mathbb{C}^*)^n$ on $\mathbb{C}^{n+2} = {\rm Spec}(\mathbb{C}[x_1,y_1,\ldots,y_n,x_n])$ given by these degrees, i.e.
\begin{multline*}
(t_1,\ldots,t_n) \cdot (x_1,y_1,\ldots,y_n,x_n) = (t_1 x_1, \ t_1^{-2}t_2 y_1, \ t_1t_2^{-2}t_3y_2, \ \ldots, \\ \ldots, \ t_{n-2}t_{n-1}^{-2}t_ny_{n-1}, \ t_{n-1}t_n^{-2}y_n, \ t_nx_n).
\end{multline*}

Indeed, the ring of invariants $\mathbb{C}[x_1,y_1,\ldots,y_n,x_1]^T$ is the subalgebra generated by $$z_1 = y_1 y_2^2 \cdots y_n^n x_n^{n+1}, z_2 = x_1^{n+1} y_1^n y_2^{n-1} \cdots y_n, \mbox{ and } w = x_1 y_1 y_2 \cdots y_n x_n,$$ which is clearly isomorphic to $\mathbb{C}[Z_1,Z_2,W]/(Z_1 Z_2 - W^{n+1})$, the coordinate ring of $\mathbb{A}_n$.

We will now show what we obtain by mimicking this construction for the non-toric du Val singularities. As we shall see, when treating the $\mathbb{D}_n$ singularities, there are some differences between odd $n$ and even $n$, so we treat them separately. We also treat independently $\mathbb{E}_6$, $\mathbb{E}_7$ and $\mathbb{E}_8$ because of their exceptional nature. Nevertheless, it will be apparent after these discussions that the resulting candidates for the Cox rings will all be analogous independently of these distinctions.

\section{Singularities of type $\mathbb{D}$.}

\subsection{Guessing the candidates.}

We focus on $\mathbb{D}$ singularities. They are defined for $n \geq 4$ as the surfaces $$\mathbb{D}_n = \{x^2 + z y^2 + z^{n-1}=0\} \subset \mathbb{C}^3.$$ The dual graph of its resolution is:

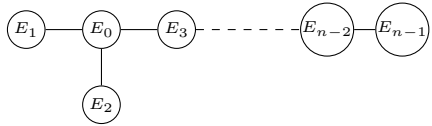
\begin{center}
\begin{tikzpicture}[scale=1]
\tikzstyle{every node}=[draw,circle,fill=white,minimum size=5mm,inner sep=0pt]
\path (180:1cm) node (e1) {\tiny $E_1$};
\path (270:1cm) node (e2) {\tiny $E_2$};
\path (0:0cm) node (e0) {\tiny $E_0$};
\path (0:1cm) node (e3) {\tiny $E_3$};
\path (0:3cm) node (en2) {\tiny $E_{n-2}$};
\path (0:4cm) node (en1) {\tiny $E_{n-1}$};
\draw 
(e0) -- (e1)
(e0) -- (e2)
(e0) -- (e3)
(en2) -- (en1);
\draw[dashed] (e3) -- (en2);
\end{tikzpicture}
\end{center}

Thus, the exceptional curve consist of $n$ (-2)-curves $E_0,E_1,\ldots,E_{n-1}$, where $E_0$ is the only one intersecting three curves, $E_1$ and $E_2$ intersect only $E_0$, and the remaining ones form a chain such that $E_3$ also intersects $E_0$. The (relative) Picard group of $\widetilde{\mathbb{D}_n}$ is isomorphic (via de degree map) to $\mathbb{Z}^n = \mathbb{Z}\langle e_0,e_1,\ldots,e_{n-1}\rangle$, where we have shifted the indices of the canonical basis of $\mathbb{Z}^n$ so that $\deg(L) = \sum_{i=0}^{n-1} \deg(L_{|E_i})$.

As the case of the $\mathbb{A}_n$ singularities suggests, we consider the polynomial ring $\bar{R} = \mathbb{C}[x_1,x_2,x_{n-1},y_0,y_1,\ldots,y_{n-1}]$, with one variable $y_i$ for each exceptional component $E_i$, and three further variables $x_1,x_2$ and $x_{n-1}$ corresponding to the leafs of the dual graph. As above, the degree of these variables is given by the {\em extended intersection matrix}
\begin{equation}
\label{matrix_Dn} \left(
\begin{array}{ccc|cccccccc}
& & & -2 & 1 & 1 & 1 & & & & \\
1 & & & 1 & -2 & & & & & & \\
& 1 & & 1 & & -2 & & & & & \\
& & & 1 & & & -2 & 1 & & & \\
& & & & & & 1 & -2 & \ddots & & \\
& & & & & & & \ddots & \ddots & \ddots & \\
& & & & & & & & \ddots & -2 & 1 \\
& & 1 & & & & & & & 1 & -2
\end{array}
\right)
\end{equation}
where the first three columns are the degrees of $x_1,x_2$ and $x_{n-1}$ respectively, and the rest of the matrix is just the intersection matrix of the exceptional curve. More explicitly, we set
{\small
\begin{eqnarray*}
d(x_1) & = & e_1 = (0,1,0,\ldots,0), \\
d(x_2) & = & e_2 = (0,0,1,0,\ldots,0), \\
d(x_{n-1}) & = & e_{n-1} = (0,\ldots,0,1), \\
d(y_0) & = & -2e_0+e_1+e_2+e_3 = (-2,1,1,1,0,\ldots,0) \\
d(y_1) & = & e_0-2e_1 = (1,-2,0,\ldots,0), \\
d(y_2) & = & e_0-2e_2 = (1,0,-2,0,\ldots,0), \\
d(y_3) & = & e_0-2e_3+e_4 = (1,0,0,-2,1,0,\ldots,0), \\
d(y_k) & = & e_{k-1}-2e_k+e_{k+1} = (0,\ldots,0,1,-2,1,0,\ldots,0) \mbox{ for $k=4,\ldots,n-2$}, \mbox{ and} \\
d(y_{n-1}) & = & e_{n-2}-2e_{n-1} = (0,\ldots,0,1,-2).
\end{eqnarray*}
}

Finally, we consider again the action of $T = {\rm Hom}({\rm Pic}(\widetilde{\mathbb{D}_n}/\mathbb{D}_n),\mathbb{C}^*) \cong (\mathbb{C}^*)^n$ on $\mathbb{C}^{n+3} = {\rm Spec}(\mathbb{C}[x_1,x_2,x_{n-1},y_0,y_1,\ldots,y_{n-1}])$ given by the degrees:
\begin{multline}
\label{action_Dn} (t_0,\ldots,t_{n-1}) \cdot (x_1,x_2,x_{n-1},y_0,\ldots,y_{n-1}) = \\
= (t_1 x_1, t_2 x_2, t_{n-1} x_{n-1}, t_0^{-2}t_1t_2t_3 y_0, t_0t_1^{-2} y_1, t_0t_2^{-2}y_2, t_0t_3^{-2}t_4 y_3, \ldots \\
\ldots, t_{n-2}t_{n-1}^{-2}y_{n-1}).
\end{multline}

We compute now the ring of invariants $\mathbb{C}[x_1,x_2,x_{n-1},y_0,y_1,\ldots,y_{n-1}]^T$ with respect to this action, which turns out to depend only on the parity of $n$. We treat two cases separately.

\begin{lemma} \label{lem_inv_D2k}
The ring of invariants of $\mathbb{C}[x_1,x_2,x_{n-1},y_0,y_1,\ldots,y_{n-1}]$ for $n=2k$ with respect to the action given by (\ref{action_Dn}) is isomorphic to $$\mathbb{C}[Z_1,Z_2,Z_3,W]/(W^2-Z_1Z_2Z_3).$$ Furthermore, the isomorphism is given by
\begin{eqnarray*}
Z_1 & = & x_1^2 y_0^{2k-2} y_1^k y_2^{k-1} y_3^{2k-3} y_4^{2k-4} \cdots y_{2k-1} \\
Z_2 & = & x_2^2 y_0^{2k-2} y_1^{k-1} y_2^k y_3^{2k-3} y_4^{2k-4} \cdots y_{2k-1} \\
Z_3 & = & x_{2k-1}^2 y_0^2 y_1 y_2 y_3^2 y_4^2 \cdots y_{2k-1}^2 \\
W & = & x_1 x_2 x_{2k-1} y_0^{2k-1} y_1^k y_2^k y_3^{2k-2} y_4^{2k-3} \cdots y_{2k-1}^2
\end{eqnarray*}
\end{lemma}
\begin{proof}
The ring of invariants is generated by monomials of degree $0$, that is, elements $m=x_1^{b_{1}}x_2^{b_{2}}x_{2k-1}^{b_{2k-1}}y_0^{a_{0}}\cdots y_{2k-1}^{a_{2k-1}}$ such that $a_i,b_i \geq 0$ and $$a_1+a_2+a_3-2a_0=0,a_{i-1}-2a_i+a_{i+1}=0,\ldots.$$ Denoting $a=b_{2k-1}, b=a_{2k-1}-b_{2k-1}, c=b_2$, and using the previous equations on the $a_i,b_i$, we get that
\begin{multline}
m=(x_1^{-1}x_2^{1}x_{2k-1}^{1}y_0^{1}y_2^{1}y_3^{1}\cdots y_{n-1}^{1})^a\cdot \\
\cdot(x_1^{-(n-2)}x_2^{n}y_0^{n-2}y_2^{n-1}y_3^{n-3}\cdots y_{n-1}^{1})^b (x_1^{2}x_2^{-2}y_1^{1}y_2^{-1})^c.
\end{multline}
Hence, the cone of monomials of degree $0$ is (isomorphic to) the subcone of $\mathbb{Z}^3 = \{(a,b,c)\}$ satisfying that all the exponents are non-negative. So looking at the exponent of $x_{2k-1}$ we get that $a\geq 0$, again looking at the exponent of $y_1$ we get that $c\geq 0$. Looking at the exponent of $x_2$ and $x_1$ we get that $-2c+nb+a\geq 0$ and $2c-(n-2)b-a\geq 0$, so $b\geq 0$. Now looking at all exponents we get that $m$ is a true monomial (its exponents are non-negative) if and only if $a-c+(n-1)b\geq 0$, $-2c+nb+a\geq 0$ and $2c-(n-2)b-a\geq 0$. The first inequality follows from the second and the third, so all the exponents are non-negative if and only if
$$a\geq 0, \, b\geq0, \, c\geq0, -2c+nb+a\geq 0 \quad \text{and} \quad 2c-(n-2)b-a\geq 0.$$
The triple $(2,0,1)$, corresponding to the monomial $Z_3$, belongs to this cone, and we take it as one of its generators. The elements from the cone which are not divisible by $(2,0,1)$ satisfy $a=0,1$ or $c=0$. If $c=0$ then from the last inequality we have $-(n-2)b-a\geq 0$, but $n\geq 4$ so $a=b=0$ and it is the origin. If $a=0$ we have $\frac{n}{2}b\geq c\geq(\frac{n}{2}-1)b$ and if $a=1$ then $\frac{n}{2}b+\frac{1}{2}\geq c\geq(\frac{n}{2}-1)b+\frac{1}{2}$. To get the first ones it is enough to add $(0,1,k-1),(0,1,k)$ as generators, and adding $(1,1,k)$ we also obtain the second ones. These three generators correspond to $Z_1,Z_2$ and $W$. Since the four generators only satisfy the relation $W^2=Z_1Z_2Z_3$, the proof is finished.
\end{proof}

Thus, the quotient of $\mathbb{C}^{2k+3}$ by the action in (\ref{action_Dn}) is the affine 3-fold in $\mathbb{C}^4$ given by the equation $W^2-Z_1Z_2Z_3=0$. We claim that intersecting it with the hypersurface $\{Z_1+Z_2+Z_3^{k-1}=0\}$ we obtain a surface isomorphic to the $\mathbb{D}_{2k}$ singularity. Indeed, $\mathbb{D}_{2k}$ can be defined as the surface $\{z^2+(y^{2k-2}-x^2)y=0\} \subset \mathbb{C}^3$. We can write the equation as $z^2+(y^{k-1}-x)(y^{k-1}+x)y=0$, or equivalently $W^2+X_1X_2X_3=0$, where $X_1+X_2=2X_3^{k-1}$, and a linear change of coordinates gives the desired expression $$\mathbb{D}_4 \cong \{W^2-Z_1Z_2Z_3 = Z_1+Z_2+Z_3^{k-1} = 0\} \subset \mathbb{C}^4.$$

Now, substituting the monomials of Lemma \ref{lem_inv_D2k} in the relation $Z_1+Z_2+Z_3^{k-1}=0$, that is, pulling back this equation via the quotient map $$\mathbb{C}^{n+3} \longrightarrow \mathbb{C}^{n+3}/T \cong \{W^2-Z_1Z_2Z_3=0\} \subset \mathbb{C}^4,$$
we obtain
$$y_0^{2k-2} y_1^{k-1} y_2^{k-1} y_3^{2k-3} y_4^{2k-4} \cdots y_{2k-1}(y_1 x_1^2 + y_2 x_2^2 + y_3 y_4^2 \cdots y_{2k-1}^{2k-3} x_{2k-1}^{2k-2}) = 0$$
which suggest that we can obtain the singularity $\mathbb{D}_{2k}$ as the quotient of the hypersurface
$$\{y_1 x_1^2 + y_2 x_2^2 + y_3 y_4^2 \cdots y_{2k-1}^{2k-3} x_{2k-1}^{2k-2}=0\}\subset\mathbb{C}^{n+3}$$
(by the same action of $T$). Indeed, this hypersurface is invariant under the action, and the remaining components of the preimage of $\{Z_1+Z_2+Z_3^{k-1}=0\}$ (given by the other factors in the pull-back) are mapped to the origin by the quotient map.

Summing up, all these computations suggest that $${\rm Cox}(\mathbb{D}_{2k}) \cong \mathbb{C}[x_1,x_2,x_{2k-1},y_0,y_1,\ldots,y_{2k-1}]/(y_1 x_1^2 + y_2 x_2^2 + y_3 y_4^2 \cdots y_{2k-1}^{2k-3} x_{2k-1}^{2k-2}),$$ which is actually true, as we shall prove later.

We take care now of the odd case.

\begin{lemma} \label{lem_inv_D2k+1}
The ring of invariants of $\mathbb{C}[x_1,x_2,x_{n-1},y_0,y_1,\ldots,y_{n-1}]$ for $n=2k+1$ with respect to the action given by (\ref{action_Dn}) is isomorphic to
\begin{multline}
\mathbb{C}[Z_1,Z_2,Z_3,Z_4,Z_5,Z_6]/(Z_2^4-Z_5Z_6,Z_1Z_2^2-Z_3Z_4,Z_2^2Z_4-Z_3Z_6,\\
Z_2^2Z_3-Z_4Z_5,Z_4^2-Z_1Z_6,Z_3^2-Z_1Z_5).
\end{multline}
Furthermore, the isomorphism is given by
\begin{eqnarray*}
Z_1 & = & x_{2k}^2 y_0^2 y_1 y_2 y_3^2 y_4^2 \cdots y_{2k}^2 \\
Z_2 & = & x_1 x_2 y_0^{2k-1} y_1^k y_2^k y_3^{2k-2} y_4^{2k-3} \cdots y_{2k} \\
Z_3 & = & x_2^2 x_{2k} y_0^{2k} y_1^{k} y_2^{k+1} y_3^{2k-1} y_4^{2k-2} \cdots y_{2k}^2 \\
Z_4 & = & x_1^2 x_{2k} y_0^{2k} y_1^{k+1} y_2^{k} y_3^{2k-1} y_4^{2k-2} \cdots y_{2k}^2 \\
Z_5 & = & x_2^4 y_0^{4k-2} y_1^{2k-1} y_2^{2k+1} y_3^{4k-4} y_4^{4k-6} \cdots y_{2k}^2 \\
Z_6 & = & x_1^4 y_0^{4k-2} y_1^{2k+1} y_2^{2k-1} y_3^{4k-4} y_4^{4k-6} \cdots y_{2k}^2.
\end{eqnarray*}
\end{lemma}
\begin{proof}
The first part of the proof of Lemma \ref{lem_inv_D2k} also applies in this case to give $(2,0,1)$ as one of the generators of the cone of monomials. However, in this case the triples not divisible by $(2,0,1)$ are those satisfying $\frac{n}{2}b\geq c\geq(\frac{n}{2}-1)b$ if $a=0$ and $\frac{n}{2}b+\frac{1}{2}\geq c\geq(\frac{n}{2}-1)b+\frac{1}{2}$ if $a=1$. To get the first ones we need to add $(0,1,k),(0,2,2k+1),$ and $(0,2,2k-1)$, and $(1,1,k)$ and $(1,1,k+1)$ are enough to obtain all the second set. This way we get the six monomials $Z_1,\ldots,Z_6$ generating the ring of invariants, and it is easy to check that the six relations given generate all the possible relations.
\end{proof}

In order to use this result to get a candidate for ${\rm Cox}(\mathbb{D}_{2k+1})$, we need to realize $\mathbb{D}_{2k+1}$ as a subvariety of $$V=\left\{\begin{array}{c} Z_2^4-Z_5Z_6=Z_1Z_2^2-Z_3Z_4=Z_2^2Z_4-Z_3Z_6= \\ =Z_2^2Z_3-Z_4Z_5=Z_4^2-Z_1Z_6=Z_3^2-Z_1Z_5=0\end{array}\right\} \subset \mathbb{C}^6.$$ First of all, the equation of $\mathbb{D}_{2k+1}$ can be rewritten as $(z^k-x)(z^k+x)+y^2z=0$, or equivalently $AB^2-CD=0$ with $C+D+A^k=0$. Hence, $\mathbb{D}_{2k+1}$ can be obtained as the intersection of $\{C+D+A^k=0\}$ with $\{AB^2-CD=0\}$. The second hypersurface turns out to be the projection of $V$ to $\mathbb{C}^4$ given by $(Z_1,\ldots,Z_6) \mapsto (A,B,C,D) = (Z_1,\ldots,Z_4)$, which suggests that $\mathbb{D}_{2k+1}$ could be contained in the intersection of $V$ with the hypersurface
\begin{equation} \label{eq-1}
Z_1^k+Z_3+Z_4=0.
\end{equation}
Indeed, this intersection consists of three irreducible components, one of which is $\mathbb{D}_{2k+1}$. In order to get rid of the two extra components, we can add the equations
\begin{equation} \label{eq-2}
Z_1^{k-1}Z_3+Z_2^2+Z_5=Z_1^{k-1}Z_4+Z_2^2+Z_6=0.
\end{equation}
Substituting the expressions of the $Z_i$ in (\ref{eq-1}) and (\ref{eq-2}), and removing the common factors as in the even case, we obtain the relation $$y_1 x_1^2 + y_2 x_2^2 + y_3 y_4^2 \cdots y_{2k}^{2k-2} x_{2k}^{2k-1},$$ which is analogous to the one obtained before.

\subsection{Constructing the (iso)morphism.}

Up to now we have got just reasonable guesses of the Cox rings of the $\mathbb{D}_n$ singularities, but we are still quite far from a proof. In order to prove the isomorphism of graded rings $${\rm Cox}(\mathbb{D}_n) \cong \mathbb{C}[x_1,x_2,x_{n-1},y_0,y_1,\ldots,y_{n-1}]/(y_1 x_1^2 + y_2 x_2^2 + y_3 y_4^2 \cdots y_{n-1}^{n-3} x_{n-1}^{n-2}),$$ we will construct a morphism $\Phi$ of graded rings and then prove that it is an isomorphism on each degree. We will define $\Phi$ by giving the images of the variables $x_1,\ldots,y_{n-1}$ and checking that the relation maps to 0.

First of all, fix once and for all, line bundles $L_0,\ldots,L_{n-1}$ on $\widetilde{\mathbb{D}_n}$ such that $\deg_{E_j}(L_{i|E_j}) = \delta_{ij}$, i.e. such that their isomorphism classes give the canonical basis of ${\rm Pic}(\widetilde{\mathbb{D}_n}/\mathbb{D}_n) \cong \mathbb{Z}^n$. Given any line bundle $L$ of degree $(d_0,\ldots,d_{n-1})$, there is a unique isomorphism $L \cong L_0^{\otimes d_0} \otimes \ldots \otimes L_{n-1}^{\otimes d_{n-1}}$, and we will always implicitly assume that we are working with the latter, even if for simplicity we write the former. All this formalism seems useless, but it is necessary in order to have a well defined multiplication in $${\rm Cox}(\mathbb{D}_n) = \sum_{(d_0,\ldots,d_{n-1})\in\mathbb{Z}^n} H^0(\widetilde{\mathbb{D}_n},L_0^{\otimes d_0} \otimes \ldots \otimes L_{n-1}^{\otimes d_{n-1}}).$$

We start with the images of the $y_i$. By construction, each exceptional component $E_i$ defines a line bundle $\mathcal{O}_{\widetilde{\mathbb{D}_n}}(E_i)$ of degree $d(y_i)$, so $\Phi(y_i)$ should belong to $H^0(\widetilde{\mathbb{D}_n},\mathcal{O}_{\widetilde{\mathbb{D}_n}}(E_i))$. Thus, we define $\Phi(y_i)$ to be the unique section (up to scalar multiplication) $s_i \in H^0(\widetilde{\mathbb{D}_n},\mathcal{O}_{\widetilde{\mathbb{D}_n}}(E_i))$ vanishing exactly along $E_i$.

In order to define $\Phi(x_1),\Phi(x_2)$ and $\Phi(x_{n-1})$ we need to choose sections $t_1 \in H^0(\widetilde{\mathbb{D}_n},L_1), t_2 \in H^0(\widetilde{\mathbb{D}_n},L_2)$ and $t_{n-1} \in H^0(\widetilde{\mathbb{D}_n},L_{n-1})$ respectively. But this choice cannot be arbitrary, since we want them to verify the relation $$s_1 t_1^2 + s_2 t_2^2 + s_3 s_4^2 \cdots s_{n-1}^{n-3} t_{n-1}^{n-2} = 0$$ as a section of $L_0$ (because the degree of this expression is $(1,0,\ldots,0)$). At this point we need to remember where the relation $y_1 x_1^2 + y_2 x_2^2 + y_3 y_4^2 \cdots y_{n-1}^{n-3} x_{n-1}^{n-2} = 0$ did come from, which depends on the parity of $n$.

In the even case $n=2k$, it was obtained as a factor of the pull-back of $Z_1+Z_2+Z_3^{k-1}$. But $x_1$ appears only in the pull-back of $Z_1$, and $Z_1=0$ defines (set-theoretically) an affine line $\overline{C_1} \subset \mathbb{D}_{2k}$ whose strict transform $C_1 \subset \widetilde{\mathbb{D}_{2k}}$ is still an affine line and is a divisor of degree $(0,1,0,\ldots,0)$. Hence, we define $t_1$ as the section of $L_1$ vanishing exactly along $C_1$. Analogously, we define $t_2$ and $t_{n-1}$ as the sections of $L_2$ and $L_{n-1}$ vanishing along the strict transforms of the affine lines in $\mathbb{D}_{2k}$ given by $Z_2=0$ and $Z_3=0$, respectively.

In the odd case, we found that $\mathbb{D}_{2k+1} \cong V \cap \{Z_1^k+Z_3+Z_4=0\} \subset \mathbb{C}^6$. Therefore, we have to take $t_1$, $t_2$ and$t_{n-1}$ the sections of the corresponding line bundles vanishing along the strict transforms of the affine lines defined by $Z_4=0$, $Z_3=0$ and $Z_1=0$ respectively.

With these choices, the relation maps to 0 by construction, and we have a well-defined homomorphism of graded rings $$\Phi: \mathbb{C}[x_1,x_2,x_{n-1},y_0,y_1,\ldots,y_{n-1}]/(y_1 x_1^2 + y_2 x_2^2 + y_3 y_4^2 \cdots y_{n-1}^{n-3} x_{n-1}^{n-2}) \longrightarrow {\rm Cox}(\mathbb{D}_n)$$ as wanted. Moreover, also by construction, its piece of degree $(0,\ldots,0)$ is an isomorphism, which is the basic step in the inductive procedure which we will use to prove that $\Phi$ is an isomorphism in every degree.

\subsection{Reduction to the non-negative case.}

As a first step, we will prove that $\Phi$ is an isomorphism on every degree if and only if it is so on all non-negative degrees (that is, those with no negative component). Note that these degrees correspond to (relatively) nef line bundles.

In order to lighten notation, we will denote $$R=\mathbb{C}[x_1,x_2,x_{n-1},y_0,y_1,\ldots,y_{n-1}]/(y_1 x_1^2 + y_2 x_2^2 + y_3 y_4^2 \cdots y_{n-1}^{n-3} x_{n-1}^{n-2}),$$ and given any divisor $D$ on $\widetilde{\mathbb{D}_n}$ or any line bundle $L \in {\mathrm{Pic}}(\widetilde{\mathbb{D}_n}/\mathbb{D}_n)$, we denote by $R_D$ or $R_L$ and ${\rm Cox}(\mathbb{D}_n)_D$ or ${\rm Cox}(\mathbb{D}_n)_L$, respectively, the summands of degree $\delta(D)$ or $\delta(L)$ of the corresponding rings. We will often identify a divisor with its associated line bundle, writing $H^0(\widetilde{\mathbb{D}_n},D)$ instead of $H^0(\widetilde{\mathbb{D}_n},\mathcal{O}_{\widetilde{\mathbb{D}_n}}(D))$.

Furthermore, in order to simplify the exposition, we make the following
\begin{definition}
Two divisors (line bundles, degrees) $D,D'$ are said to be {\em equivalent} if the following conditions are equivalent:
\begin{itemize}
\item $\Phi_D: R_D \rightarrow Cox(X)_D$ is an isomorphism
\item $\Phi_{D'}:R_{D'} \rightarrow Cox(X)_{D'}$ is an isomorphism
\end{itemize}
\end{definition}

Thus, our next objective is to show that each degree is equivalent a non-negative one. We start with a preliminary lemma (recall that $E_i$ denote the components of the exceptional divisor, and that $\delta_i(D) = \deg_{E_i}(\mathcal{O}_{E_i}(D)) = D\cdot E_i$).

\begin{lemma}\label{red1}
Let $D$ be a divisor such that $d_i=\delta_i(D)<0$. Then $D$ and $D-E_i$ are equivalent.
\end{lemma}
\begin{proof}
Let us consider the following exact sequence:
{\small
$$
\xymatrix{
0 \ar[r] & R_{D-E_i} \ar^{\cdot y_i}[r] \ar[d]^{\Phi_{D-E_i}} & R_{D} \ar[r] \ar[d]^{\Phi_{D}} & Q \ar[r] \ar@{-->}[d]^{\cong} & 0 \\
0 \ar[r] & H^0(\widetilde{\mathbb{D}_n},D-E_i) \ar^{\cdot s_i}[r] & H^0(\widetilde{\mathbb{D}_n},D) \ar[r] & H^0(E_i,\mathcal{O}_{E_i}(d_i)) = 0 \ar[r] & 0}
$$
}
We want to see that $\Phi_{D-E_i}$ is an isomorphism if and only if $\Phi_D$ is. On the one hand, We have that $H^0(E_i,\mathcal{O}_{E_i}(d_i)) = 0$ because $d_i<0$. On the other hand, $Q=0$ because multiplying by $y_i$ is a surjection. Indeed, in the case $i=4,\ldots,n-1$, if $x_1^{\beta_1}x_2^{\beta_2}x_3^{\beta_3}y_1^{\alpha_1}\cdots y_n^{\alpha_n}\in R_{D}$ then $\alpha_{i-1}-2\alpha_i+\alpha_{i+1}=d_i<0$, so $\alpha_i>0$ (all the exponents are non-negative), which means that it is the image of $x_1^{\beta_1}x_2^{\beta_2}x_3^{\beta_3}y_1^{\alpha_1}\cdots y_i^{\alpha_i-1}\cdots y_n^{\alpha_n}\in R_{D-E_i}$. The rest of the cases are analogous, so we have proved the assertion.
\end{proof}

We are now ready to prove the next
\begin{proposition} \label{prop-red-1}
For every divisor $D$ there exists a nef divisor $D'$ (i.e., such that $\delta_i(D')\geq 0$ for all $i$) equivalent to $D$.
\end{proposition}
\begin{proof}
Order the $E_i$'s as $E_1<E_2<E_0<E_3<\cdots<E_n$ and proceed in the following way:
\begin{itemize}
\item if $\delta(D)$ has some negative component, choose the lowest (according to the order above) negative component $i$ and replace $D$ by $D-E_i$.
\item if $\delta(D)$ has no negative component, stop the procedure.
\end{itemize}
Due to Lemma \ref{red1} it is enough to prove that the above procedure stops after finitely many steps. Suppose that for the divisor $D$ it does not. Let $A = a_1,a_2,a_3,\ldots$ be the infinite sequence of choices of the lowest negative coordinate in the steps of the procedure, and let $d^1,d^2,d^3,\ldots$ be the sequence of degrees. Let $j$ be the highest coordinate (always according to the order fixed above) that appears infinitely many times in the sequence $A$. There exists some $n_0$ such that $a_n \leq j$ for $n>n_0$ (since higher indices appear only finitely many times). Let $n_1$ be greater than $n_0$ and such that $a_{n_1}=j$ (it is possible since $j$ appears infinitely many times). So $d^{n_1}$ has nonnegative coordinates for $i<j$ and $d^{n_1}_j<0$ so $d^{n_1}=(nn,nn,\dots,nn,d^{n_1}_j,\dots)$ ($nn$ denotes a nonnegative number and $p$ denotes a positive number). Let us observe what happens in the next steps of procedure. If $d^{n_1}_i>0$ for some $i=0,3,\dots,j-1$ and it is the highest such $i$ then
$d^{n_1}=(nn,nn,\dots,p,0,\dots,0,d^{n_1}_j,\dots)$

$d^{n_1+1}=(nn,nn,\dots,p,0,\dots,0,-1,d^{n_1}_j+2,\dots)$

$d^{n_1+2}=(nn,nn,\dots,p,0,\dots,0,-1,1,d^{n_1}_j+1,\dots)$

$d^{n_1+3}=(nn,nn,\dots,p,0,\dots,0,-1,1,0,d^{n_1}_j+1,\dots)$

...

$d^{n_1+j-i-1}=(nn,nn,\dots,p,-1,1,0,\dots,0,d^{n_1}_j+1,\dots)$

$d^{n_1+j-i}=(nn,nn,\dots,p-1,1,0,\dots,0,d^{n_1}_j+1,\dots)$

\noindent so all indices lower then $j$ are nonnegative and $j$-th is greater by $1$. The same happens then

$d^{n_1}=(0,\dots,0,d^{n_1}_j,\dots)$

$d^{n_1+1}=(0,\dots,0,-1,d^{n_1}_j+2,\dots)$

$d^{n_1+2}=(0,\dots,0,-1,1,d^{n_1}_j+1,\dots)$

$d^{n_1+3}=(0,\dots,0,-1,1,0,d^{n_1}_j+1,\dots)$

...

$d^{n_1+j-3}=(0,0,0,-1,1,0,\dots,0,d^{n_1}_j+1,\dots)$

$d^{n_1+j-2}=(-1,0,0,1,0,\dots,0,d^{n_1}_j+1,\dots)$

$d^{n_1+j-1}=(1,-1,-1,0,0,\dots,0,d^{n_1}_j+1,\dots)$

$d^{n_1+j}=(-1,1,1,0,0,\dots,0,d^{n_1}_j+1,\dots)$

$d^{n_1+j+1}=(1,0,0,0,\dots,0,d^{n_1}_j+1,\dots)$

Hence after finitely many steps we will have $d^m_j\geq 0$, and for all $i<j$, $d^m_i\geq 0$ so $a_m>j$, which is a contradiction.
\end{proof}

\subsection{Reduction to basic cases.}

Now that we know that it is enough to check that $\Phi$ is an isomorphism on nef degrees, we want to reduce this problem to check only a few (basic) degrees. As before, we need first some results studying the relation between $\Phi_D$ and $\Phi_{D'}$ when adding or subtracting elementary divisors. However, the situation now is a bit more complicated because we should always have nef degrees.

\begin{lemma}\label{red2}
Let $D$ be a nef divisor such that $d_i=\delta_i(D)\geq 2$ for some $i$. Then $D+E_i$ is nef and equivalent to $D$.
\end{lemma}
\begin{proof}
Let us consider the following exact sequence:
{\footnotesize
\begin{equation} \label{Exact-seq}
\xymatrix{
0 \ar[r] & R_{D} \ar^{\cdot y_i}[r] \ar[d]^{\Phi_D} & R_{D+E_i} \ar[r] \ar[d]^{\Phi_{D+E_i}} & Q \ar[r] \ar@{-->}[d]^{\cong} & 0 \\
0 \ar[r] & H^0(\widetilde{\mathbb{D}_n},D) \ar^{\cdot s_i}[r] & H^0(\widetilde{\mathbb{D}_n},D+E_i) \ar[r] & H^0(E_i,\mathcal{O}_{E_i}(d_i-2)) \ar[r] & H^1(\widetilde{\mathbb{D}_n},D)}
\end{equation}
}
On the one hand, since $D$ is relatively nef (by hypothesis) and relatively big (because $\pi$ is birational), and $\mathbb{D}_n$ is affine, then 
$$H^1(\widetilde{\mathbb{D}_n},D)=H^0(\mathbb{D}_n,R^1\pi_*D)=0$$
because of the relative Kawamata-Viehweg vanishing theorem. On the other hand, the multiplication by $y_i$ in the upper row is of course injective, since the relation in $R$ is irreducible. We are going to describe the generators of $Q$ over $\mathbb{C}$ and observe that it has dimension equal to the dimension of $H^0(E_i,\mathcal{O}_{E_i}(d_i-2))$, which is $d_i-1$. The monomials of $R_{D+E_i}$ which are not in the image of $R_{D}$ are of the form $m = x_1^{b_1}x_2^{b_2}x_{n-1}^{b_{n-1}}y_0^{a_0}\cdots y_{n-1}^{a_{n-1}}$ with $a_i=0$. We can assume moreover, thanks to the relation, that $y_1x_1^2$ does not divide $m$, so either $a_1=0$ or $b_1\leq 0$. They have to be of degree $\delta(D+E_i)$ so $-2a_i+a_{i+1}+a_{i-1}=a_{i+1}+a_{i-1}=d_i-2$ (it is the case $i=3,\dots,n-1$, the remaining ones are analogous), so we have $d_i-1$ possibilities: $a_{i-1}=a$ and $a_{i+1}=d_i-2-a$ for $a=0,\dots,d_i-2$. It is enough to show that each possibility can be extended to a monomial in a unique way. Before going through the vertex $E_0$ it is easy because the degree force us to have $a_{i-j}=ja_{i-1}+(j-1)d_{i-1}+\dots+d_{i-j+1}$. Then we have either $a_1+a_2+a_3-2a_0=d_0$ and $a_0>a_3$ or $a_0=a_3=0$. In the last case the only solutions are of the form
$$(a_1,b_1,a_2,b_2)=(c,2c+d_1,d_0-c,2d_0-2c+d_2).$$
In the case $a_1$ the only possibility is $c=0$, for which all components are non-negative. Otherwise, if $a_1>0$ we must have $b_1=0,1$ according with the parity of $d_1$, and in each case there is only one possibility for $c$. On the other hand, in the first case the solutions are of the form
{\small
\begin{multline*} 
(a_1,b_1,a_2,b_2)=\left(\left\lfloor \frac{1}{2}(2a_0-a_3+d_0)\right\rfloor+c,a_0-\left\lfloor \frac{1}{2}(2a_0-a_3+d_0)\right\rfloor+d_1+2c,\right.\\
\left.\left\lceil \frac{1}{2}(2a_0-a_3+d_0)\right\rceil-c,a_0-\left\lceil \frac{1}{2}(2a_0-a_3+d_0)\right\rceil+d_2-2c\right)
\end{multline*}
}
and again there is only one possible $c$ such that $y_1x_1^2\not| \, \, m$.
\end{proof}

As we said before, the situation now is more complicated and it is not enough for our purposes to add single exceptional components to the divisor. Indeed, we will need to add chains of exceptional components in order to reach a basic case having only nef divisors along the procedure, so we need the following lemma.

\begin{lemma}\label{red3}
Let $D$ be a divisor such that $d_i = \delta_i(D) = 1$, $d_j = \delta_j(D) \geq 1$ and $\delta_k(D)=0$ for all $E_k$ between $E_i$ and $E_j$, and let $E=E_i+\cdots+E_j$ be the sum of the components corresponding to the path joining $E_i$ and $E_j$. Assume moreover that $d_j = 1$ if $E_j$ is not a leaf of the dual graph. Then $D+E$ is nef and equivalent to $D$.
\end{lemma}
\begin{proof}
First of all, notice that since the dual graph of the resolution is a tree, there is only one path joining $E_i$ and $E_j$ and hence $E$ is well defined.

Let us consider now the following exact sequence:
{\footnotesize
$$
\xymatrix{
0 \ar[r] & R_{D} \ar^{\cdot y_i\cdots y_j}[r] \ar[d]^{\Phi_{D}} & R_{D+E} \ar[r] \ar[d]^{\Phi_{D+E}} & Q \ar[r] \ar@{-->}[d]^{\cong} & 0 \\
0 \ar[r] & H^0(\widetilde{\mathbb{D}_n},D) \ar^{\cdot s_i\cdots s_j}[r] & H^0(\widetilde{\mathbb{D}_n},D+E) \ar[r] & H^0(E,\mathcal{O}_E(D+E)) \ar[r] & H^1(\widetilde{\mathbb{D}_n},D)}
$$
}
As in the previous proof, $H^1(X,D)=0$ due to the relative Kawamata-Viehweg vanishing theorem. In the upper row multiplication by $y_i\cdots y_j$ is again an injection, since the relation in $R$ is irreducible. We are going to describe the generators of $Q$ over $\mathbb{C}$ and observe that its dimension equals the dimension of $H^0(E,\mathcal{O}_E(D+E))$. The rest of the cases being analogous, we will give the details assuming $j > i \geq 3$. In this case, $\dim H^0(E,\mathcal{O}_E(D+E))=d_i+d_j-1=d_j$. The monomials of $R_{D+E}$ which are not in the image of $R_D$ are of the form $x_1^{b_1}x_2^{b_2}x_3^{b_3}y_0^{a_0}\cdots y_{n-1}^{a_{n-1}}$ where at least one of the $a_i,\dots,a_j$ is equal to $0$. They have to be of degree $\delta(D+E)$ so $a_{k-1}-2a_k+a_{k+1}=0$ for $k=i,\dots,j-1$, hence if $a_k=0$ for some of these $k$, then $(a_i,\dots,a_j)=(0,\dots,0)$. Otherwise, it must hold $a_j=0$ and $a_{j-1}>0$, and since $a_{j-1}-2a_j+a_{j+1}=d_j-1$, there are exactly $d_j-1$ possibilities ($a_{j-1}=1,2,\ldots,d_j-1$) each of which can be extended to a monomial in a unique way. These possibilities, together with the first one, make a total of $d_j$ possible monomials, as wanted.
\end{proof}

We are ready to prove the next reduction result.

\begin{proposition}  \label{prop-red-2}
Every divisor $D$ is equivalent to a divisor $D'$ of degree $d'\in\{0,ke_1=(0,k,0,\dots,0),ke_2=(0,0,k,0,\dots,0),e_{n-1}=(0,\dots,0,k) \,|\, k>0\}$.
\end{proposition}
\begin{proof}
First we will prove that there is an equivalent divisor $D'$ with all degrees equal to $0$ except at most one, which is $1$.

We proceed in the following way:
\begin{itemize}
\item when we have some coordinate $\delta_i(D)\geq 2$ then we replace $D$ by $D+E_i$,
\item when we have all coordinates $\delta_i(D)\in\{0,1\}$ and at least two $1$'s on coordinates $i,j$ and
$0$'s between them, then we replace $D$ by $D+E$, where $E=E_i+\cdots+E_j$,
\item when $D$ has at most one nonzero coordinate which is equal to $1$ we stop.
\end{itemize}

Firstly, let us observe that applying the above procedure to a nonnegative degree we stay in the nonnegative case. Secondly, due to Lemmas \ref{red2} and \ref{red3}, it is enough to prove that the above procedure stops after finitely many steps, because all divisors produced by this procedure are equivalent.

Suppose that for some divisor $D$ it is not the case. Let $A=a_1,a_2,a_3,\ldots$ be the infinite sequence of sets of indices of the $E_i$'s added in the above procedure in the consecutive steps, and let $d^1,d^2,d^3,\ldots$ be the consecutive degrees. We define the new order on the $E_i$'s: $E_1<E_2<E_0<E_3<\cdots<E_n$. Let $j$ be the highest coordinate (according to the new order) that appears infinitely many times in the sets in $A$. There exists some $n_0$ such that for $n>n_0$ $a_n$ contains no element greater then $j$ (since higher indices appear only finitely many times). Given any degree $d$, let us consider the sum
$$S=\frac{1}{2}(d_1+d_2)+\sum_{i\neq 1,2}d_i.$$
Observe that in the steps of the above procedure the value of $S$ does not increase and after a step such that $j \in a_k$ it decreases by $1$, additionally this sum is nonnegative. It is a contradiction with the fact that $j$ appears infinitely many times as an element of $a_i$.

After having reached a divisor $D'$ with all multidegrees equal to 0 but one equal to 1 (say $\delta_i(D')$) it is easy to get an equivalent divisor of the form $kE_j$ with $j=1,2,n-1$. Indeed, take $E_j$ the final component in the branch from $E_0$ containing $E_i$ (if $i=0$, pick any of them), and apply subsequently Lemma \ref{red3} in the other way, i.e. with $D'$ as $D+E$. This way we move the 1 one step each time towards $E_j$, and each time the coefficient of $E_j$ increases by one, hence we will end with a divisor of the form we wanted.
\end{proof}

\subsection{Reduction to $0$-graded piece.}

The results of the previous sections allow us to check if $\Phi_D$ is an isomorphism for every $D \in {\rm Pic}(\widetilde{\mathbb{D}_n}/\mathbb{D}_n) \cong \mathbb{Z}^n$ just checking it for the basic degrees $\{0,ke_1,ke_2,ke_{n-1} \,|\, k>0\}$. Since we already know that $\Phi_0$ is an isomorphism (by construction), we will reduce the remaining three cases to this (most basic) one.

\begin{proposition}  \label{prop-red-3}
Every degree $d \in \{ke_1,ke_2,ke_{n-1} \,|\, k>0\}$ is equivalent to $0$.
\end{proposition}
\begin{proof}
All cases being analogous, we will consider the case $d=ke_1$ for some $k > 0$ and show that it is equivalent to $(k-1)e_1$, which is clearly enough.

First of all, recall that $C_1$ is an affine line intersecting transversely just $E_1$, and that $t_1$ is the section of $L_1$ vanishing exactly along $C_1$. Therefore we can consider the exact sequence of sheaves
$$0 \longrightarrow L_1^{\otimes(k-1)} \stackrel{\cdot t_1}{\longrightarrow} L_1^{\otimes k} \longrightarrow L_{1|C_1}^{\otimes k} \cong \mathcal{O}_{C_1}
\longrightarrow 0$$
where the last isomorphism holds because $C_1$ is affine.

As in the proofs of the previous results, there is an induced commutative diagram
{\scriptsize
$$
\xymatrix{ 0 \ar[r] & R_{(k-1)e_1} \ar^{\cdot x_1}[r] \ar[d]^{\Phi_{(k-1)e_1}} & R_{ke_1} \ar[r] \ar[d]^{\Phi_{ke_1}} & Q \ar[r] \ar@{-->}[d]^{\cong} & 0 \\
0 \ar[r] & H^0(\widetilde{\mathbb{D}_n}, L_1^{\otimes(k-1)}) \ar^{\cdot t_1}[r] & H^0(\widetilde{\mathbb{D}_n}, L_1^{\otimes k}) \ar[r] & H^0(C_1,\mathcal{O}_{C_1}) \ar[r] & H^1(\widetilde{\mathbb{D}_n}, L_1^{\otimes(k-1)})=0 
}$$
}
where as usual the first row is exact because the relation in $R$ is irreducible and $Q$ is defined as the corresponding cokernel, and the second row is obviously exact as well.

It only remains to see that the rightmost map is an isomorphism. In this case, $Q$ is the piece of degree $ke_1$ of the quotient
$$ R/x_1R \cong \mathbb{C}[x_2,x_{n-1},y_0,\ldots,y_{n-1}]/(y_2x_2^2+y_3\cdots y_{n-1}^{n-3}x_{n-1}^{n-2}).$$
Therefore, a basis of $Q$ is given by the monomials $M = x_2^{b_2}x_{n-1}^{b_{n-1}}y_0^{a_0}\cdots y_{n-1}^{a_{n-1}}$ of degree $ke_1$ which are not divisible by $y_2x_2^2$. The condition $\deg(M) = ke_1$ is equivalent to the linear system of equations in the exponents
$$\begin{array}{l} 0 = -2a_0+a_1+a_2+a_3 \\ k = a_0-2a_1 \\ 0 = a_0-2a_2+b_2 \\ 0 = a_0-2a_3+a_4 = a_3-2a_4+a_5 = \ldots = a_{n-2}-2a_{n-1}+b_{n-1} \end{array}$$
which easily implies that both $a_0$ and $b_2$ must have the same parity as $k$. Both cases being analogous, we will proof the case $k$ odd.

Firstly, $M$ is not divisible by $y_2x_2^2$ if and only if either $a_2 = 0$ or $b_2 \leq 1$. Since $2a_2 = a_0+b_2 \geq 2$ implies $a_2 > 0$,  it must hold $b_2 \leq 1$, and therefore $b_2=1$. Now, the solutions of the first three equations are given by
$$b_2=1, \quad a_2=a_1+1, \quad a_0=2a_1+k \quad \text{and} \quad a_3=2a_1+\frac{3k-1}{2}$$
and the last row of equations implies that $a_0,a_3,a_4,\ldots,a_{n-1},b_{n-1}$ is an arithmetic progression of difference $a_3-a_0=\frac{k-1}{2} \geq 0$. Therefore, since all the solutions are non-negative if $a=a_{-1}\geq0$, $Q$ has a countable basis given by the monomials
$$M_a=\left(x_2x_{n-1}^{1+\frac{n(k-1)}{2}}y_0^ky_2^{\frac{k+1}{2}}y_3^{\frac{3(k-1)}{2}+1}\cdots y_{n-1}^{\frac{(n-1)(k-1)}{2}+1}\right)(x_{n-1}^2y_0^2y_1y_2y_3^2\cdots y_{n-1}^2)^a.$$

On the other hand $H^0(C_1,\mathcal{O}_{C_1}) \cong \mathbb{C}[T]$, where $T$ is any function on $C_1$ vanishing of order 1 at $C_1 \cap E_1$ (or any other point). Following the diagram we see that the image of $M_a$ is $$u(t_{n-1}^2s_0^2s_1s_2s_3^2\cdots s_{n-1}^2)_{|C_1}^a$$ where $u$ is an invertible function depending only on the isomorphism $L_{1|C_1}^{\otimes k} \cong \mathcal{O}_{C_1}$ and $t=(s_0^2s_1s_2s_3^2\cdots s_{n-1}^2t_{n-1}^2)_{|C_1} \in H^0(C_1,\mathcal{O}_{C_1})$ vanishes of order 1 at $C_1 \cap E_1$ because of the factor $s_1$ (the rest of factors do not vanish at any point of $C_1$). So we can divide by $u$ take $T=t$ to get $Q \cong \mathbb{C}[T]$, finishing the proof.
\end{proof}

We are now able to state the main result for the $\mathbb{D}_n$ singularities.

\begin{theorem}
The Cox ring of the $\mathbb{D}_n$ singularity (following Definition \ref{def-Cox}) is isomorphic to
$$\mathbb{C}[x_1,x_2,x_{n-1},y_0,y_1,\ldots,y_{n-1}]/(y_1x_1^2+y_2x_2^2+y_3y_4^2 \cdots y_{n-1}^{n-3}x_{n-1}^{n-2}).$$
\end{theorem}
\begin{proof}
It is a simple consequence of Propositions \ref{prop-red-1}, \ref{prop-red-2} and \ref{prop-red-3}.
\end{proof}

\begin{remark} \label{rem-rule}
It is worth noting that the relation $y_1x_1^2+y_2x_2^2+y_3y_4^2 \cdots y_{n-1}^{n-3}x_{n-1}^{n-2}$ can be easily read from the dual graph of the singularity. Indeed, consider the triple node $E_0$ as the root, and assign to each branch (extended with the extra variable $x_i$) the monomial on the corresponding variables with the distance to $E_0$ as exponents. For example, the branch consisting of $E_1$ gives the monomial $y_1x_1^2$, and so on. Then, the relation is obtained simply by adding these monomials. This method also works in the $\mathbb{E}$ case, as will be shown in the last section.
\end{remark}

\begin{remark}
There is a good geometrical reason to justify that ${\rm Cox}(\mathbb{D}_n)$, must have a relation of the kind above: For each of the components $E_1$, $E_2$ and $E_3$ intersecting $E_0$, there is a section $s_i$ in the Cox ring vanishing exactly along them. Their restrictions to $E_0$ are sections of $H^0(E_0,\mathcal{O}_{E_0}(1)) \cong \mathbb{C}^2$, so they have to be linearly dependent. This dependence relation, once adjusted so that it defines a global section of some line bundle (this is where the rest of the variables come from) gives a relation in the Cox ring similar to the one we obtained. Moreover, this argument shows that in the Cox ring of any singularity there must be at least one relation for each node in the dual graph with degree greater than 2.
\end{remark}

\section{Singularities of type $\mathbb{E}$ and open questions.}

In this last section, we will expose the results about type $\mathbb{E}$ singularities which lead to the explicit computation of its Cox ring. We will also give some example for which the rule explained in Remark \ref{rem-rule} does not apply, and a short list of open questions and next steps to do in the study of Cox rings of more general surface singularities.

The (affine) $\mathbb{E}_n$ singularities are defined for $n=6,7,8$ as the hypersurfaces
$$\mathbb{E}_6 = \{x^4+y^3+z^2=0\} \subseteq \mathbb{C}^3$$
$$\mathbb{E}_7 = \{x^3y+y^3+z^2=0\} \subseteq \mathbb{C}^3$$
$$\mathbb{E}_8 = \{x^5+y^3+z^2=0\} \subseteq \mathbb{C}^3$$
and the dual graphs of its minimal resolution is of the form

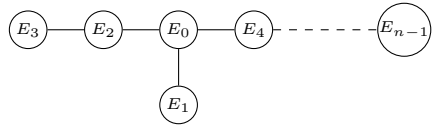
\begin{center}
\begin{tikzpicture}[scale=1]
\tikzstyle{every node}=[draw,circle,fill=white,minimum size=5mm,inner sep=0pt]
\path (180:2cm) node (x0) {\tiny $E_3$};
\path (180:1cm) node (x1) {\tiny $E_2$};
\path (270:1cm) node (x2) {\tiny $E_1$};
\path (0:0cm) node (y0) {\tiny $E_0$};
\path (0:1cm) node (y1) {\tiny $E_4$};
\path (0:3cm) node (yn1) {\tiny $E_{n-1}$};
\draw 
(x0) -- (x1)
(x1) -- (y0)
(x2) -- (y0)
(y0) -- (y1);
\draw[dashed] (y1) -- (yn1);
\end{tikzpicture}
\end{center}

\begin{lemma}
Numbering the nodes as in the figure, considering variables $y_i$ for each $E_i$ and $x_1,x_3,x_{n-1}$ as usual, and considering the action of $T=\mathbb{C}^{*n}$ given by the corresponding (extended) intersection matrix, the rings of invariants are as given in the table, monomials giving the isomorphisms are also included.


{
\begin{tabular}{|l|l|l|}
\hline
$n$ & Ring of invariants & Isomorphism \\
\hline
 &  & $Z_1=y_0^3y_1^2y_2^2y_3y_4^2y_5x_1$ \\
6 & $\mathbb{C}[Z_1,Z_2,Z_3,Z_4]/(Z_2^3-Z_3Z_4)$ & $Z_2=y_0^4y_1^2y_2^3y_3^2y_4^3y_5^2x_3x_5$ \\
 &  & $Z_3=y_0^6y_1^3y_2^4y_3^2y_4^5y_5^4x_3^3$ \\
 &  & $Z_4=y_0^6y_1^3y_2^5y_3^4y_4^4y_5^2x_5^3$ \\
\hline
 &  & $Z_1=y_0^4y_1^2y_2^3y_3^2y_4^3y_5^2y_6x_3$ \\
7 & $\mathbb{C}[Z_1,Z_2,Z_3,Z_4]/(Z_2^2-Z_3Z_4)$ & $Z_2=y_0^{12}y_1^7y_2^8y_3^4y_4^9y_5^6y_6^3x_1^2$ \\
 &  & $Z_3=y_0^9y_1^5y_2^6y_3^3y_4^7y_5^5y_6^3x_1x_6$ \\
 &  & $Z_4=y_0^6y_1^3y_2^4y_3^2y_4^5y_5^4y_6^3x_6^2$ \\
\hline
 &  & $Z_1=y_0^{15}y_1^8y_2^{10}y_3^5y_4^{12}y_5^9y_6^6y_7^3x_1$ \\
8 & $\mathbb{C}[Z_1,Z_2,Z_3]$ & $Z_2=y_0^6y_1^3y_2^4y_3^2y_4^5y_5^4y_6^3y_7^2x_7$ \\
 &  & $Z_3=y_0^{10}y_1^5y_2^7y_3^4y_4^8y_5^6y_6^4y_7^2x_3$ \\
\hline
\end{tabular}
}
\end{lemma}

As for the $\mathbb{D}_n$ cases, we need to intersect the variety corresponding to the previous rings with some hypersurface in order to obtain (surfaces isomorphic to) the singularities. The choices are summarized in the next
\begin{lemma}
\begin{itemize}
\item $\mathbb{E}_6$ is isomorphic to the intersection of $\{Z_2^3-Z_3Z_4=0\}$ with $H_6=\{Z_1^2+Z_3+Z_4=0\}$.
\item $\mathbb{E}_7$ is isomorphic to the intersection of $\{Z_3^2-Z_2Z_4=0\}$ with $H_7=\{Z_1^3+Z_2+Z_4^2=0\}$.
\item $\mathbb{E}_8$ is isomorphic to $H_8=\{Z_2^5+Z_3^3+Z_1^2=0\}$.
\end{itemize}
\end{lemma}

Substituting in the equations of $H_n$ the corresponding expressions of the $Z_i$ in terms of $x_i$ and $y_i$, and then cutting common factors of the $y_i$, we obtain in the three cases the expression $$y_1x_1^2+y_2y_3^2x_3^3+y_4y_5^2\cdots y_{n-1}^{n-4}x_{n-1}^{n-3}=0,$$ which is analogous to those obtained in the previous cases. Since the reduction results (lemmas and propositions in the previous sections) are still true (with minor changes) in these three cases, we conclude with next

\begin{theorem}
The Cox ring of the $\mathbb{E}_n$ singularity (following Definition \ref{def-Cox}) is isomorphic to
$$\mathbb{C}[x_1,x_2,x_{n-1},y_0,y_1,\ldots,y_{n-1}]/(y_1x_1^2+y_2y_3^2x_3^3+y_4y_5^2\cdots y_{n-1}^{n-4}x_{n-1}^{n-3}).$$
\end{theorem}

\begin{remark} \label{counterex}
The rule of Remark \ref{rem-rule} is not valid of every singularity, even those with just one trivalent node. For example, consider the singularity $X$ obtained as the contraction of the configuration of $(-2)$-curves shown in the next figure, with three branches of lenght 2. In this case, the Cox ring is not isomorphic to the analogous candidate $\mathbb{C}[x_2,x_4,x_6,y_0,\ldots,y_6]/(y_1y_2^2x_2^3+y_3y_4^2x_4^3+y_5y_6^2x_6^3)$. Indeed, what is not true are the reduction Lemmas. For example, we will show that Lemma \ref{red2} does not hold. Suppose we have a divisor $D$ of degree $2e_7$ (according to the numbering shown in the figure) and we try to reduce to the degree $e_6$ by adding $E_7$. If we look at the two cokernels in the diagram (\ref{Exact-seq}), the second one is $H^0(E_7,\mathcal{O}(0) \cong \mathcal{C}$, and so $Q$ should be one dimensional. However, it is easy to see that there is no monomial of degree $e_6$ with no $y_7$ (so that it does not belong to the image of multiplication by $y_7$). The reason is that the other leaves are ``too long'', and then the conditions of the exponents to be non-negative are too restrictive.
\end{remark}

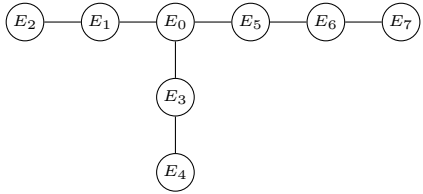
\begin{center}
\begin{tikzpicture}[scale=1]
\tikzstyle{every node}=[draw,circle,fill=white,minimum size=5mm,inner sep=0pt]
\path (180:1cm) node (x1) {\tiny $E_1$};
\path (180:2cm) node (x0) {\tiny $E_2$};
\path (270:1cm) node (x2) {\tiny $E_3$};
\path (270:2cm) node (x3) {\tiny $E_4$};
\path (0:0cm) node (y0) {\tiny $E_0$};
\path (0:1cm) node (y1) {\tiny $E_5$};
\path (0:2cm) node (y2) {\tiny $E_6$};
\path (0:3cm) node (x4) {\tiny $E_7$};
\draw 
(x0) -- (x1)
(y2) -- (x4)
(x1) -- (y0)
(x2) -- (y0)
(x3) -- (x2)
(y0) -- (y1)
(y1) -- (y2);
\end{tikzpicture}
\end{center}

We finish the paper with some open questions, all concerning more complicated kinds of singularities, which could lead to generalizations of our results.

\begin{question}
Compute Cox rings of Hirzebruch-Jung singularities (see \cite{BHPVV}, section III.5), which are singularities whose dual graph is still a line, like in the $\mathbb{A}_n$ case, but some of the exceptional curves have self-intersection smaller than $-2$.
\end{question}

\begin{question}
Compute Cox rings of singularities all whose exceptional curves are still $(-2)$-curves, but whose dual graph is not a Dynkin diagram. For example,
\begin{enumerate}
\item if the tree has just a trivalent node, what happens if there is no branch of lenght 1? (as in Remark \ref{counterex})
\item what happens if there are more than one trivalent node in the dual graph, but still no node with higher valency?
\item what happens if there is still one node with valency greater than 2, but it has valency 4,5,6,...?
\item what happens in general case?
\end{enumerate}
\end{question}



\section*{Acknowledgements}


We would like to thank very much Jaros\l aw Wi\'sniewski and Luis Sol\'a-Conde for introducing us to this
subject. We are also grateful for many useful and interesting talks. We would also like to thank Maria Donten-Bury, who carefully read the published version of this paper, pointing out and solving some small mistakes.

\end{document}